\documentclass[a4,10pt]{article}
\usepackage{graphicx,url,amsmath,amssymb,amsthm,palatino,subfigure,siunitx,tikz,pgfplots,anysize}
\usepackage[affil-it]{authblk}
\marginsize{2cm}{2cm}{2cm}{2cm}

\newcommand{\ka}{k_{\mathrm{a}}}
\newcommand{\epsa}{\varepsilon_{\mathrm{a}}}
\newcommand{\epsb}{\varepsilon_{\mathrm{b}}}
\newcommand{\mub}{\mu_{\mathrm{b}}}
\newcommand{\mua}{\mu_{\mathrm{a}}}	
\newcommand{\sigmab}{\sigma_{\mathrm{b}}}
\newcommand{\set}[1]{\left\{#1\right\}}
\newcommand{\abs}[1]{\left|#1\right|}
\newcommand{\p}{\partial}
\newcommand{\rd}{\mathrm{d}}

\newcommand{\mx}{\mathbf{x}}
\newcommand{\mz}{\mathbf{z}}

\newcommand{\vn}{\boldsymbol{\nu}}
\newcommand{\vt}{\boldsymbol{\theta}}

\DeclareMathOperator*{\sing}{sing}
\DeclareMathOperator*{\reg}{reg}

\DeclareMathOperator*{\leng}{length}

\DeclareMathOperator*{\dist}{dist}

\theoremstyle{plain}
\newtheorem{theorem}{Theorem}[section]
\newtheorem{lemma}{Lemma}[section]

\theoremstyle{remark}

\newtheorem{example}{Example}[section]
\newtheorem{property}{Property}[section]

\begin{document}
\title{Topological derivative for a fast identification of short, linear perfectly conducting cracks with inaccurate background information}
\author{Won-Kwang Park}
\affil{Department of Information Security, Cryptology, and Mathematics, Kookmin University, Kookmin University, Seoul, 02707, Korea. Electronic address: parkwk@kookmin.ac.kr}
\date{}
\maketitle
\begin{abstract}
In this study, we consider a topological derivative-based imaging technique for the fast identification of short, linear perfectly conducting cracks completely embedded in a two-dimensional homogeneous domain with smooth boundary. Unlike conventional approaches, we assume that the background permittivity and permeability are unknown due to their dependence on frequency and temperature, and we propose a normalized imaging function to localize cracks. Despite inaccuracies in background parameters, application of the proposed imaging function enables to recognize the existence of crack but it is still impossible to identify accurate crack locations. Furthermore, the shift in crack localization of imaging results is significantly influenced by the applied background parameters. In order to theoretically explain this phenomenon, we show that the imaging function can be expressed in terms of the zero-order Bessel function of the first kind, the crack lengths, and the applied inaccurate background wavenumber corresponding to the applied inaccurate background permittivity and permeability. Various numerical simulations results with synthetic data polluted by random noise validate the theoretical results.
\end{abstract}

\section{Introduction}\label{sec:1}
Accurate and efficient identification of unknown defects from scattering phenomena by applying microwaves is of paramount importance in various engineering and physical applications, including non-destructive testing \cite{HRMGAC,PGBM,WWXHSZLF}, structural health monitoring \cite{CM2,HSM2,IV}, biomedical imaging \cite{A1,A2,CZBN}, and radar technology \cite{BMKS,C2,TKNN}. Due to the fundamental nonlinearity and ill-posedness of the problem, various iterative (quantitative) inversion techniques have been developed to address microwave imaging challenges. Examples include Newton-type shape reconstruction for arc-like cracks \cite{K}, Gauss–Newton method for breast cancer imaging \cite{RMMP}, Born iterative approach for brain stroke detection \cite{IBA}, level set method for retrieving shape of unknown scatterers \cite{DL}, Newton–Kantorovich method for arm and thorax imaging \cite{MJBB}, and optimal control approach for shape reconstruction of extended scatterers \cite{AGJKLY}. Although these iterative-based inversion techniques have been successfully applied to various real-world microwave imaging problems, several critical issues such as non-convergence phenomena, entrapment in local minima, and significant computational costs due to the large number of iterations, may arise if the iterative process is initiated without a good initial guess and prior information of unknown targets. Moreover, because iterative-based techniques are highly sensitive to measurement noise, this can lead to substantial errors in the reconstructed results so that appropriate regularization strategies must be adopted. Hence, for obtaining a good initial guess or prior information (location, shape, etc.) of unknown objects, various non-iterative (qualitative) inversion techniques have been investigated.

Among the many available non-iterative techniques, topological derivative-based imaging methods have drawn significant attention. Originally, this was originally investigated for the shape optimization problem \cite{ASSS,CR1,CGGM,EKS,NS,SZ}, but its application for fast identification of unknown objects in inverse scattering problem has been demonstrated successfully; for example, identification of small defects \cite{B1,BGN}, fast imaging of thin electromagnetic inhomogeneities \cite{P-TD3,P-TD1}, shape imaging of three dimensional obstacles \cite{LR1}, flaw detection in welding joints \cite{MR}, object localization from Fresnel experimental dataset \cite{CPR}, and infrared thermography for defect identification inside thin plates \cite{PR}. We also refer to \cite{AGJK,BHR,CDLR,CR3,HKO} for various remarkable results about the topological derivatives.

In general, the effectiveness of topological derivative methods hinges on the accurate knowledge of background material properties—specifically, the permittivity, permeability, and conductivity of the medium. However, in real-world applications, such parameters are often inaccurately known due to their sensitivity to environmental variations such as temperature and operating frequency. Related works can be found in \cite{BWFWH,C8,K7,MPFLP,P-Book} and we refer to Figure \ref{Parameters}. This leads to critical issues in practical implementations, where even slight deviations from true background values can inaccurate imaging result. However, to the best of our knowledge, a reliable mathematical theory has not been investigated to explain this phenomenon.

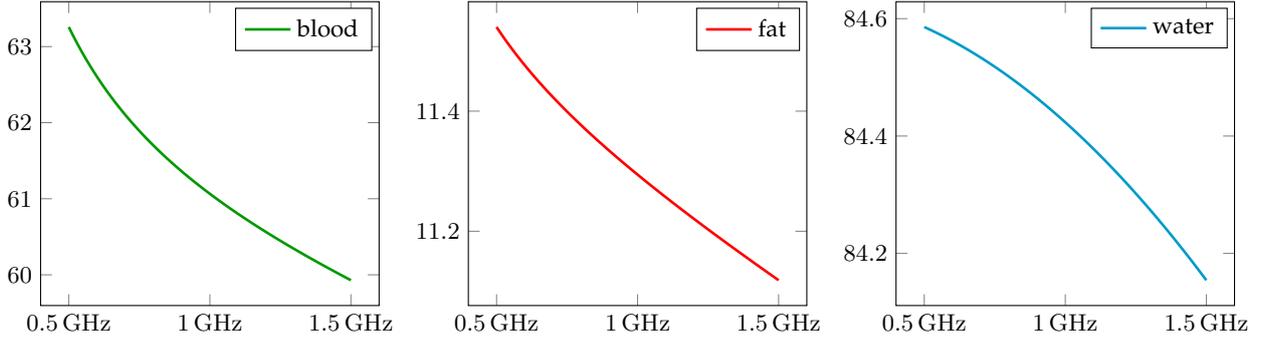
\begin{figure}[h]
\begin{center}
\begin{tikzpicture}
\small
\begin{axis}
[xtick={1,501,1001},
xticklabels={$\SI{0.5}{\giga\hertz}$,$\SI{1}{\giga\hertz}$,$\SI{1.5}{\giga\hertz}$},
width=.355\textwidth,
height=.33\textwidth,
legend cell align={left}]
\addplot[line width=1pt,solid,color=green!60!black] %
	table[x=x,y=y3,col sep=comma]{Permittivity.csv};
\addlegendentry{blood};
\end{axis}
\end{tikzpicture}
\begin{tikzpicture}
\small
\begin{axis}
[xtick={1,501,1001},
xticklabels={$\SI{0.5}{\giga\hertz}$,$\SI{1}{\giga\hertz}$,$\SI{1.5}{\giga\hertz}$},
width=.355\textwidth,
height=.33\textwidth,
legend cell align={left}]
\addplot[line width=1pt,solid,color=red] %
	table[x=x,y=y2,col sep=comma]{Permittivity.csv};
\addlegendentry{fat};
\end{axis}
\end{tikzpicture}
\begin{tikzpicture}
\small
\begin{axis}
[xtick={1,501,1001},
xticklabels={$\SI{0.5}{\giga\hertz}$,$\SI{1}{\giga\hertz}$,$\SI{1.5}{\giga\hertz}$},
width=.355\textwidth,
height=.33\textwidth,
legend cell align={left}]
\addplot[line width=1pt,solid,color=cyan!80!black] %
	table[x=x,y=y1,col sep=comma]{Permittivity.csv};
\addlegendentry{water};
\end{axis}
\end{tikzpicture}
\caption{\label{Parameters}Permittivity values of blood (left), fat (center), and water (right) between $f=\SI{0.5}{\giga\hertz}$ and $f=\SI{1.5}{\giga\hertz}$.}
\end{center}
\end{figure}

To address this issue, we consider the application of topological derivative for identifying a set of two-dimensional short, linear perfectly conducting cracks completely embedded in a homogeneous medium, even when accurate background permittivity or permeability values are unknown. To this end, we introduce a topological derivative based imaging function based on a misfit functional with inaccurate permittivity and permeability values. In order to explain some phenomena including the appearance of inaccurate location of cracks, we carefully investigate the structure of the imaging function by establishing its relationship with crack length, both accurate and inaccurate wavenumbers, and the zero-order Bessel function. This is based on the fact that the boundary measurement data can be expressed by the asymptotic expansion formula in the presence of small cracks. To validate the theoretical results,
numerical simulation results with boundary measurement data corrupted by random noise in the existence of single and multiple short, linear cracks are exhibited.

This paper is organized as follows: In Section \ref{sec:2}, we introduce the basic concept of the forward problem in the presence of a set of short, linear perfectly conducting cracks and design topological derivative based imaging function with inaccurate background wavenumber corresponding to the applied permittivity and permeability values. In Section \ref{sec:3}, we explore mathematical structure of the designed imaging function and discuss its various properties including the appearance of inaccurate results. In Section \ref{sec:4}, we  present numerical simulation results to validate the theoretical result. In Section \ref{sec:5}, a short conclusion including future directions is provided.

\section{Topological derivative based imaging}\label{sec:2}
Let $\Gamma_m$ be a linear perfectly conducting crack with length $2\ell_m$ and center $\mx_m$ and $\Gamma$ denotes the collection of $\Gamma_m$. Throughout this paper, all $\Gamma_m$ are completely embedded in a homogeneous material $\Omega\subset\mathbb{R}^2$ with smooth boundary $\p\Omega\in\mathcal{C}^2$, permittivity $\epsb$, and permeability $\mub$ at given angular frequency $\omega=2\pi f$. Throughout this paper, we assume that $\Omega$ is a subset of anechoic chamber; thus, $\mub=4\pi\times\SI{e-7}{\henry/\meter}$, $\epsb=\SI{8.854e-12}{\farad/\meter}$, and the background conductivity $\sigmab\approx0$. By denoting $k=\omega\sqrt{\epsb\mub}$ and $\lambda$ as the background wavenumber and given positive wavelength, respectively, we assume that
\begin{equation}\label{condition}
k|\mx_m-\mx_{m'}|\gg\frac14\quad\text{and}\quad\ell_m<\frac{\lambda}{2},
\end{equation}
for $m,m'=1,2,\ldots,M$ and $m\ne m'$. These conditions means that $\Gamma_m$ are sufficiently separated from each
other and the their lengths are short. Furthermore, we assume that $\Gamma_m$ does not touch the $\p\Omega$ i.e., there is a positive constant $C$ such that
\[\dist(\Gamma_m,\p\Omega)\geq C.\]

In the presence of $\Gamma$, let $u^{(n)}(\mx;k)$ be the time-harmonic total field satisfying the boundary value problem
\begin{equation}\label{ForwardProblem}
\left\{\begin{array}{rcl}
\medskip\triangle u^{(n)}(\mx;k)+k^2 u^{(n)}(\mx;k)=0&\mbox{in}&\Omega\backslash\overline{\Gamma}\\
u^{(n)}(\mx;k)=0&\mbox{on}&\Gamma\\
\noalign{\medskip}\displaystyle\frac{\p u^{(n)}(\mx;k)}{\p\vn(\mx)}=\frac{\p\exp(ik\vt_n\cdot\mx)}{\p\vn(\mx)}=g^{(n)}(\mx)\in L^2(\partial\Omega)&\mbox{on}&\p\Omega,\\
\end{array}\right.
\end{equation}
where $\vn(\mx)$ is the outward normal to $\mx\in\p\Omega$ and $\vt_n\in\mathbb{S}^1$, $n=1,2,\ldots,N$, denotes the propagation direction. We assume that $k^2$ is not an eigenvalue of \eqref{ForwardProblem} to guarantee the well-posedness. Analogously, let $w^{(n)}(\mx;k)=\exp(ik\vt_n\cdot\mx)$ be the background solution which satisfies \eqref{ForwardProblem} in the absence of $\Gamma$.

Now, let us consider the following misfit functional, which depends on the solution $u^{(n)}(\mx;k)$:
\begin{equation}\label{Energy}
  \mathbb{E}(\Omega)=\frac12\sum_{n=1}^{N}\int_{\p\Omega}|u^{(n)}(\mx;k)-w^{(n)}(\mx;k)|^2\rd\mx.
\end{equation}
In order to introduce the topological derivative based imaging function, let us create a small linear crack $\Sigma$ with length $\epsilon$ at $\mz\in\Omega\backslash\partial\Omega$ and denote $\Omega|\Sigma$ as the corresponding domain. Since the creation of $\Sigma$ changes the topology of the entire domain $\Omega$, it is natural to consider the corresponding topological derivative $d_T\mathbb{E}(\mz)$ of $\mathbb{E}(\Omega)$ with respect to point $\mz$ such that
\begin{equation}\label{TopDerivative}
d_T\mathbb{E}(\mz)=\lim_{\gamma\to0+}\frac{\mathbb{E}(\Omega|\Sigma) -\mathbb{E}(\Omega)}{\varphi(\gamma)},
\end{equation}
where $\varphi(\gamma)\longrightarrow0$ as $\gamma\longrightarrow0+$. From (\ref{TopDerivative}), we obtain the asymptotic expansion:
\begin{equation}\label{AsymptoticExpansionTopologicalDerivative}
  \mathbb{E}(\Omega|\Sigma)=\mathbb{E}(\Omega)+\varphi(\gamma)d_T\mathbb{E}(\mz)+o(\varphi(\gamma)).
\end{equation}
Based on \cite{P-TD3}, we can explore explicitly forms of $\varphi(\gamma)$ and $d_T\mathbb{E}(\mz)$ as follows.

\begin{lemma}Let $\mathbb{E}(\Omega)$ is defined as \eqref{Energy}. Then $\varphi(\gamma)$ and $d_T\mathbb{E}(\mz)$ are written as
\begin{equation}\label{TopologicalDerivative}
\varphi(\gamma)=\frac{2\pi}{\ln(\gamma/2)}\quad\text{and}\quad d_T\mathbb{E}(\mz)=\mathrm{Re}\sum_{n=1}^{N}v^{(n)}(\mz;k)\overline{w^{(n)}(\mz;k)},
\end{equation}
respectively. Here, $v^{(n)}(\mx;k)$ satisfies the adjoint problem
\begin{equation}\label{Adjoint1}
\left\{\begin{array}{rcl}
\displaystyle\Delta v^{(n)}(\mx;k)+k^2v^{(n)}(\mx;k)=0&\mbox{in}&\Omega\\
\noalign{\medskip}\displaystyle\frac{\p v^{(n)}(\mx;k)}{\p\vn(\mx)}=u^{(n)}(\mx;k)-w^{(n)}(\mx)&\mbox{on}&\p\Omega.
  \end{array}\right.
\end{equation}
\end{lemma}

Based on prior studies \cite{AM2,B1,CPR,LR1,P-TD1,P-TD3,P-TD5}, existence and location of $\Gamma_m$ can be identified through the map of $d_T\mathbb{E}(\mz)$. However, for a successful application of $d_T\mathbb{E}(\mz)$, exact values of $\epsb$ and $\mub$ must be known. However, since the values of $\epsb$ and $\mub$ are dependent on the applied frequency and ambient temperature, it will be inappropriate to apply \eqref{TopologicalDerivative} without prior information of exact $\epsb$ and $\mub$.

Since accurate value of $k$ is unknown, let us apply alternative background wavenumber $\ka=\omega\sqrt{\epsa\mua}$ with inaccurate permittivity $\epsa$ and permeability $\mua$, and introduce the following normalized topological derivative based imaging function
\begin{equation}\label{NormalizedTopologicalDerivative}
  \mathfrak{F}(\mz;\ka)=\frac{|d_T\mathbb{E}(\mz;\ka)|}{\displaystyle\max_{\mz\in\Omega}|d_T\mathbb{E}(\mz;\ka)|}\quad\text{where}\quad d_T\mathbb{E}(\mz;\ka)=\mathrm{Re}\sum_{n=1}^{N}v^{(n)}(\mz;\ka)\overline{w^{(n)}(\mz;\ka)},
\end{equation}
where $w^{(n)}(\mx;\ka)=\exp(i\ka\vt_n\cdot\mx)$ and $v^{(n)}(\mx;\ka)$ satisfies the adjoint problem
\begin{equation}\label{Adjoint2}
  \left\{\begin{array}{rcl}
    \displaystyle\Delta v^{(n)}(\mx;\ka)+\ka^2 v^{(n)}(\mx;\ka)=0&\mbox{in}&\Omega\\
\noalign{\medskip}\displaystyle\frac{\p v^{(n)}(\mx;\ka)}{\p\vn(\mx)}=u^{(n)}(\mx;k)-w^{(n)}(\mx;k)&\mbox{on}&\p\Omega.
  \end{array}\right.
\end{equation}
Fortunately, based on simulation results in Section \ref{sec:4}, the existence of $\Gamma_m$ can be recognized through the map of $\mathfrak{F}(\mz;\ka)$ even if $\ka\ne k$ but the identified location $\mx_m\in\Gamma_m$ is shifted in a specific direction. However, to the best of our knowledge, the mathematical background to support this has not been established yet.

\section{Structure of the imaging function}\label{sec:3}
In this section, we explore mathematical structure of the topological derivative based imaging function $\mathfrak{F}(\mz;\ka)$ and discuss its various properties. To this end, we recall the asymptotic expansion formula derived in \cite{AKLP}.

\begin{lemma}[Asymptotic expansion formula]
Let $u^{(n)}(\mx;k)$ satisfies \eqref{ForwardProblem} and $w^{(n)}(\mx)$ be the background solution of \eqref{ForwardProblem}. Then, if the condition \eqref{condition} is valid, the following asymptotic expansion formula holds uniformly on $\mx\in\p\Omega$;
\begin{equation}\label{Asymptoticformula}
u^{(n)}(\mx;k)-w^{(n)}(\mx;k)=\sum_{m=1}^{M}\frac{2\pi}{\ln(\ell_m/2)}w^{(n)}(\mx_m;k)\mathcal{N}(\mx,\mx_m;k)+\mathcal{O}\left(\frac{1}{|\ln\ell_m|^2}\right),
\end{equation}
where $\mathcal{N}$ denotes the Neumann function, which satisfies
\[\left\{\begin{array}{rcl}
    \displaystyle\Delta \mathcal{N}(\mx,\mz;k)+k^2\mathcal{N}(\mx,\mz;k)=-\delta(\mx,\mz)&\mbox{in}&\Omega\\
\noalign{\medskip}\displaystyle\frac{\p\mathcal{N}(\mx,\mz;k)}{\p\vn(\mx)}=0&\mbox{on}&\p\Omega.
  \end{array}\right.\]
Here, $\delta$ is the Dirac delta function.
\end{lemma}

Now, we derive the main result.
\begin{theorem}\label{Theorem}
For sufficiently large $N$, $\mathfrak{F}(\mz;\ka)$ can be represented as follows
\begin{equation}\label{Structure_Function}
\mathfrak{F}(\mz;\ka)\approx\frac{|\Phi(\mz)|}{\displaystyle\max_{\mz\in\Omega}|\Phi(\mz)|},\quad\Phi(\mz)=\sum_{m=1}^{M}\frac{J_0(|k\mx_m-\ka\mz|)}{\ln(\ell_m/2)},
\end{equation}
where $J_0$ is the Bessel function of order zero of the first kind.
\end{theorem}
\begin{proof}
Since $v^{(n)}(\mx;\ka)$ satisfies \eqref{Adjoint2}, for $\mz\in\Omega\backslash\p\Omega$, we have
\begin{equation}\label{FormulaV}
v^{(n)}(\mz;\ka)=\int_{\p\Omega}\frac{v^{(n)}(\mx;\ka)}{\p\vn(\mx)}\mathcal{N}(\mx,\mz;\ka)\rd\mx=\int_{\p\Omega}\left(u^{(n)}(\mx;k)-w^{(n)}(\mx;k)\right)\mathcal{N}(\mx,\mz;\ka)\rd\mx.
\end{equation}
Applying \eqref{Asymptoticformula} and \eqref{FormulaV} to \eqref{NormalizedTopologicalDerivative}, we can obtain
\begin{align}
\begin{aligned}\label{Term}
d_T\mathbb{E}(\mz;\ka)&=\mathrm{Re}\sum_{n=1}^{N}v^{(n)}(\mz;\ka)\overline{w^{(n)}(\mz)}\\
&\approx\mathfrak{Re}\sum_{n=1}^{N}\left[\int_{\p\Omega}\left(u^{(n)}(\mx;k)-w^{(n)}(\mx)\right)\mathcal{N}(\mx,\mz;\ka)\rd\mx\right]\overline{w^{(n)}(\mz;\ka)}\\
&\approx\mathfrak{Re}\sum_{n=1}^{N}\left[\int_{\p\Omega}\left(\sum_{m=1}^{M}\frac{2\pi}{\ln(\ell_m/2)}w^{(n)}(\mx_m;k)\mathcal{N}(\mx,\mx_m;k)\right)\mathcal{N}(\mx,\mz;\ka)\rd\mx\right]\overline{w^{(n)}(\mz;\ka)}\\
&=\mathfrak{Re}\sum_{m=1}^{M}\frac{2\pi}{\ln(\ell_m/2)}\left(\int_{\p\Omega}\mathcal{N}(\mx,\mx_m;k)\mathcal{N}(\mx,\mz;\ka)\rd\mx\right)\left(\sum_{n=1}^{N}w^{(n)}(\mx_m;k)\overline{w^{(n)}(\mz;\ka)}\right).
\end{aligned}
\end{align}

Since the following relationship holds for sufficiently large $N$ (see \cite{P-SUB3} for derivation),
\[\frac{1}{N}\sum_{n=1}^{N}\exp(\vt_n\cdot\mx)\approx J_0(|\mx|),\]
we can derive
\begin{equation}\label{Term1}
\sum_{n=1}^{N}w^{(n)}(\mx_m;k)\overline{w^{(n)}(\mz;\ka)}=\sum_{n=1}^{N}\exp\Big(\vt_n\cdot(k\mx_m-\ka\mz)\Big)=NJ_0(|k\mx_m-\ka\mz|).
\end{equation}

Next, following to \cite{AKL}, the Neumann function $\mathcal{N}$ has a logarithmic singularity. Thus, it can be decomposed into the singular and regular functions;
\[\mathcal{N}(\mx,\mz;k)=-\frac{1}{2\pi}\ln|\mx-\mz|+\mathcal{R}(\mx,\mz;k),\quad\mx\ne\mz,\]
where $\mathcal{R}(\mx,\mz;k)\in\mathcal{C}^{1,\alpha}$ in both $\mathbf{x}$ and $\mathbf{y}$ for some $\alpha$ with $0<\alpha<1$. Since $\mx\in\p\Omega$ and $\mx_m\in\Gamma_m$, there is no blow up of $\mathcal{N}(\mx,\mx_m;k)$ so that
\[|\mathcal{N}(\mx,\mx_m;k)|\leq\frac{1}{2\pi}\ln|\mx-\mx_m|+|\mathcal{R}(\mx,\mx_m;k)| <\frac{1}{2\pi}\ln\mbox{diam}(\Omega)+\max_{\mx\in\p\Omega}|\mathcal{R}(\mx,\mx_m;k)|,\]
where $\mbox{diam}(\Omega)$ denotes the diameter of $\Omega$. Hence, applying H{\"o}lder's inequality yields
\[\abs{\int_{\p\Omega}\mathcal{N}(\mx,\mx_m;k)\mathcal{N}(\mx,\mz;\ka)\rd\mx}\leq \bigg(\frac{1}{2\pi}\ln\mbox{diam}(\Omega)+\max_{\mx\in\p\Omega}|\mathcal{R}(\mx,\mx_m;k)|\bigg) \int_{\p\Omega}|\mathcal{N}(\mx,\mz;\ka)|\rd\mx.\]
Now, we consider the singularity of $\mathcal{N}(\mx,\mz;\ka)$ when $\mz\in\Omega$ is close to $\mx\in\p\Omega$. To this end, for sufficiently small constant $\rho>0$, generate a ball $B(\mx,\rho)$ with radius $\rho$ and center $\mx$ such that
\[B(\mx,\rho)\cap\Gamma=\varnothing.\]
Based on this, let us split $\p\Omega$ into two smooth curves $\p\Omega=\p\Omega_{\sing}\cup\p\Omega_{\reg}$ (see Figure \ref{Boundaries} for illustration), where
\[\p\Omega_{\sing}=\Omega\cap\p B(\mx,\rho)\quad\text{and}\quad\p\Omega_{\reg}=\p\Omega\backslash\overline{\p\Omega_{\sing}}.\]
Then, we can examine that
\begin{align*}
\int_{\p\Omega}|\mathcal{N}(\mx,\mz;\ka)|\rd\mx\leq &\frac{1}{2\pi}\int_{\p\Omega}\ln|\mx-\mz|\rd\mx +\int_{\p\Omega}|\mathcal{R}(\mathbf{z},\mathbf{y};\omega)|\rd\mx\\
\leq&\frac{1}{2\pi}\lim_{\rho\to0+}\left(\int_{\p\Omega_S}\ln|\mx-\mz|\rd\mx +\int_{\p\Omega_R}\ln|\mx-\mz|\rd\mx\right)+\max_{\mx\in\p\Omega}|\mathcal{R}(\mx,\mz;\ka)|\leng(\p\Omega)\\
\leq&\frac{1}{2\pi}\lim_{\rho\to0+}\bigg(\rho\ln\rho+(\leng(\p\Omega)-\rho)\ln|\leng(\p\Omega)|\bigg)+\max_{\mx\in\p\Omega}|\mathcal{R}(\mx,\mz;\ka)|\leng(\p\Omega)\\
=&\leng(\p\Omega)\bigg(\frac{1}{2\pi}\ln|\leng(\p\Omega)|+\max_{\mx\in\p\Omega}|\mathcal{R}(\mx,\mz;\ka)|\bigg).
\end{align*}
Here, $\leng(\p\Omega)$ denotes the length of $\partial\Omega$. Therefore,
\begin{multline}\label{Term2}
  \abs{\int_{\p\Omega}\mathcal{N}(\mx,\mx_m;k)\mathcal{N}(\mx,\mz;\ka)\rd\mx}\leq\bigg(\frac{1}{2\pi}\ln\mbox{diam}(\Omega)+\max_{\mx\in\p\Omega}|\mathcal{R}(\mx,\mx_m;k)|\bigg) \times\\
  \leng(\p\Omega)\bigg(\frac{1}{2\pi}\ln|\leng(\p\Omega)|+\max_{\mx\in\p\Omega}|\mathcal{R}(\mx,\mz;\ka)|\bigg).
\end{multline}
This means that there is no blowup of $\int_{\p\Omega}\mathcal{N}(\mx,\mx_m;k)\mathcal{N}(\mx,\mz;\ka)\rd\mx$ for any $\mx\in\p\Omega$ and $\mz\in\Omega$.

Plugging \eqref{Term1} and \eqref{Term2} into \eqref{Term}, we can examine that
\[d_T\mathbb{E}(\mz;\ka)\propto\sum_{m=1}^{M}\frac{J_0(|k\mx_m-\ka\mz|)}{\ln(\ell_m/2)}\]
and correspondingly, \eqref{Structure_Function} can be ontained.
\end{proof}

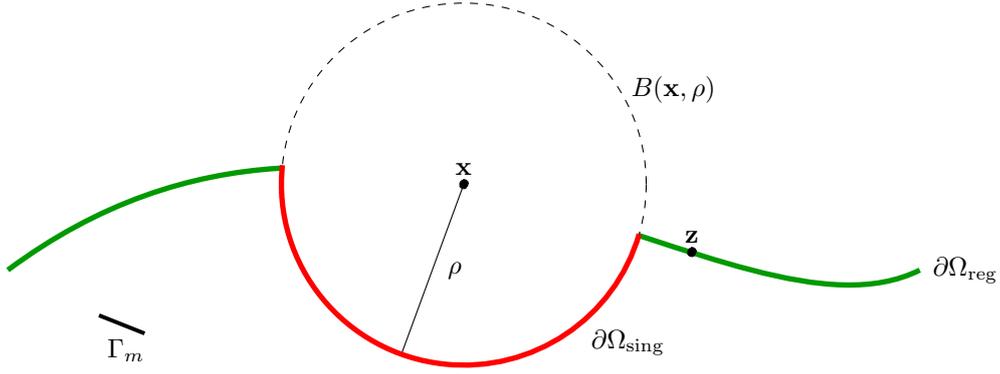
\begin{figure}[h]
\begin{center}
\begin{tikzpicture}[scale=1.2]
\def\SA{174};
\def\FA{344};
\draw[green!60!black,rounded corners,line width=2pt,fill=white] (-5,0) .. controls (-1,3) and (3,-1) .. (5,0);
\node at (5.5,0) {$\p\Omega_{\reg}$};
\draw[fill=white,dashed] (0,0.95) circle [radius=2];
\node at (2.3,2) {$B(\mx,\rho)$};
\draw[black,-] (0,0.95) -- ({2*cos(250)},{2*sin(250)+0.95});
\node at (-0.1,0) {$\rho$};
\draw[red,solid,line width=2pt] ({2*cos(\SA)},{2*sin(\SA)+0.95}) arc (\SA:\FA:2);
\node at (1.8,-0.8) {$\p\Omega_{\sing}$};
\draw[fill=black,dashed] (0,0.95) circle [radius=0.05];
\node [above] at (0,0.95) {$\mx$};
\draw[fill=black,dashed] (2.5,0.2) circle [radius=0.05];
\node [above] at (2.5,0.2) {$\mz$};
\draw[black,line width=1.5pt,-] (-4,-0.5) -- (-3.5,-0.7) node [below,xshift=-7,yshift=2] {$\Gamma_m$};
\end{tikzpicture}
\caption{\label{Boundaries}Illustration of $\p\Omega_{\sing}$ and $\p\Omega_{\reg}$.}
\end{center}
\end{figure}

Now, we discuss some properties of $\mathfrak{F}(\mz;\ka)$ based on the Theorem \ref{Theorem}.

\begin{property}[Identification of inaccurate location and shape]\label{property1}
Since $J_0(x)$ has its maximum value $1$ at $x=0$, the location $\mz=(k/\ka)\mx_m$, $m=1,2,\ldots,M$, will be identified through the map of $\mathfrak{F}(\mz;\ka)$ instead of the true location $\mx_m\in\Gamma_m$. Hence, although the existence of cracks can be recognized, their exact location cannot be determined unless the exact values of $\epsb$ and $\mub$ are known.
\end{property}

\begin{property}[Observation of shifting effect when applying inaccurate permittivity]\label{property2}
Assume that $\epsa\ne\epsb$, $\mua=\mub$, and $\ka$ satisfies \eqref{condition}. Then, for $\mx_m\in\Sigma_m$, identified location $\mz$ becomes
\[\mz=\left(\frac{k}{\ka}\right)\mx_m=\sqrt{\frac{\epsb}{\epsa}}\mx_m.\]
Hence, the retrieved location and length of $\Gamma_m$ will be close to the origin and shorter than the true $\Gamma_m$, respectively when $\epsa>\epsb$ (refer to the red-colored straight lines in Figure \ref{IllustrationResult}). Otherwise, the retrieved location and length of $\Gamma_m$ will be far-away from the origin and longer than the true $\Gamma_m$, respectively when $\epsa<\epsb$ (refer to the blue-colored straight lines in Figure \ref{IllustrationResult}). If the center of a crack $\Gamma_m$ is the origin, its accurate location can be retrieved for any nonzero $\epsa$. Same phenomenon can be examined for the case $\epsa=\epsb$ and $\mua\ne\mub$.
\end{property}

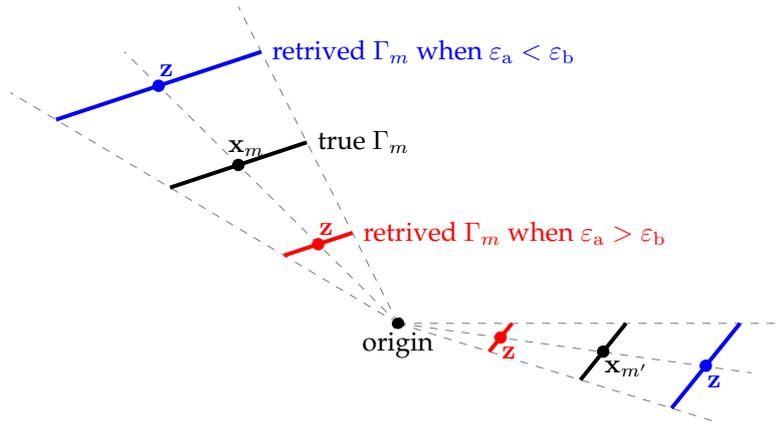
\begin{figure}[h]
\begin{center}
\begin{tikzpicture}[scale=1.5]
\draw[gray,dashed] (0,0) -- (-1.4,2.8);
\draw[gray,dashed] (0,0) -- (-2.4,2.4);
\draw[gray,dashed] (0,0) -- (-3.4,2.04);
\draw[gray,dashed] (0,0) -- (3.4,0);
\draw[gray,dashed] (0,0) -- (3.1,-0.4375);
\draw[gray,dashed] (0,0) -- (2.8,-0.875);
\draw[fill=black,dashed] (0,0) circle [radius=0.05] node[below,black] { origin};
\draw[red,line width=1.5pt,-] (-1,0.6) -- (-0.4,0.8) node [right] { retrived $\Gamma_m$ when $\epsa>\epsb$};
\draw[red,fill=red] (-0.7,0.7) circle (0.05cm) node[above,xshift=3,red] {$\mz$};
\draw[black,line width=1.5pt,-] (-2,1.2) -- (-0.8,1.6) node [right] { true $\Gamma_m$};
\draw[black,fill=black] (-1.4,1.4) circle (0.05cm) node[above,xshift=3,black] {$\mx_m$};
\draw[blue,line width=1.5pt,-] (-3,1.8) -- (-1.2,2.4) node [right] { retrived $\Gamma_m$ when $\epsa<\epsb$};
\draw[blue,fill=blue] (-2.1,2.1) circle (0.05cm) node[above,xshift=3,blue] {$\mz$};

\draw[red,line width=1.5pt,-] (0.8,-0.25) -- (1,0);
\draw[red,fill=red] (0.9,-0.125) circle (0.05cm) node[below right,xshift=-3,red] {$\mz$};
\draw[black,line width=1.5pt,-] (1.6,-0.5) -- (2,0);
\draw[black,fill=black] (1.8,-0.25) circle (0.05cm) node[below right,xshift=-3,black] {$\mx_{m'}$};
\draw[blue,line width=1.5pt,-] (2.4,-0.75) -- (3,0);
\draw[blue,fill=blue] (2.7,-0.375) circle (0.05cm) node[below right,xshift=-3,blue] {$\mz$};
\end{tikzpicture}
\caption{\label{IllustrationResult} (Property \ref{property2}) Illustration of simulation results. Black-colored straight lines are true cracks, blue- and red-colored straight lines are retrieved cracks when $\epsa<\epsb$ and $\epsa>\epsb$, respectively.}
\end{center}
\end{figure}

\section{Numerical simulation results}\label{sec:4}
In this section, we present some numerical simulation results with synthetic data to support Theorem \ref{Theorem}. For the simulation, the homogeneous domain $\Omega$ is chosen as the unit circle centered at the origin, applied frequency of operation is selected as $f=\SI{875}{\mega\hertz}$, and the incident direction $\vt_n$ is
\[\vt_n=(\cos\theta_n,\sin\theta_n)=\left(\cos\frac{2(n-1)\pi}{N},\sin\frac{2(n-1)\pi}{N}\right),\quad n=1,2,\ldots,N(=32).\]
With this configuration, the boundary measurement data in the absence and presence of small cracks, and the adjoint problem \eqref{Adjoint2} were obtained and solved by the finite element method (FEM), respectively. After the generation of the boundary measurement data, a $20$dB Gaussian random noise was added to the unperturbed data.

\begin{figure}[h]
\begin{center}
\subfigure[Examples \ref{Ex1} and \ref{Ex2}]{
\begin{tikzpicture}[scale=2.5]
\draw[gray] (0,0) circle (1);
\node at (0.8,-0.8) {$\Omega$};
\draw[black,line width=1.5pt,-] (-0.03,0) -- (0.03,0) node [right] {$\Gamma_{\mathrm{o}}$};
\draw[black,line width=1.5pt,-] (0.5788,-0.0212) -- (0.6212,0.0212) node [right] {$\Gamma_{\mathrm{x}}$};
\draw[black,line width=1.5pt,-] (-0.03,0.5) -- (0.03,0.5) node [right] {$\Gamma_{\mathrm{y}}$};
\end{tikzpicture}}\quad
\subfigure[Example \ref{Ex3}]{\begin{tikzpicture}[scale=2.5]
\draw[gray] (0,0) circle (1);
\node at (0.8,-0.8) {$\Omega$};
\draw[black,line width=1.5pt,-] (-0.63,-0.1) -- (-0.57,-0.1) node [right] {$\Gamma_1$};
\draw[black,line width=1.5pt,-] (0.2788,0.4788) -- (0.3212,0.5212) node [right] {$\Gamma_2$};
\draw[black,line width=1.5pt,-] (0.1740,-0.5850) -- (0.2260,-0.6150) node [right] {$\Gamma_3$};
\end{tikzpicture}}\quad
\subfigure[Example \ref{Ex4}]{\begin{tikzpicture}[scale=2.5]
\draw[gray] (0,0) circle (1);
\node at (0.8,-0.8) {$\Omega$};
\draw[black,line width=1.5pt,-] (-0.62,-0.1) -- (-0.58,-0.1) node [right] {$\Gamma_1$};
\draw[black,line width=1.5pt,-] (0.2576,0.4576) -- (0.3424,0.5424) node [right] {$\Gamma_2$};
\draw[black,line width=1.5pt,-] (0.2087,-0.6050) -- (0.1913,-0.5950) node [right] {$\Gamma_3$};
\end{tikzpicture}}
\caption{\label{IllustrationCracks} Illustration of simulation results. Black-colored straight lines are true cracks, blue- and red-colored straight lines are retrieved cracks when $\epsa>\epsb$ and $\epsa<\epsb$, respectively.}
\end{center}
\end{figure}
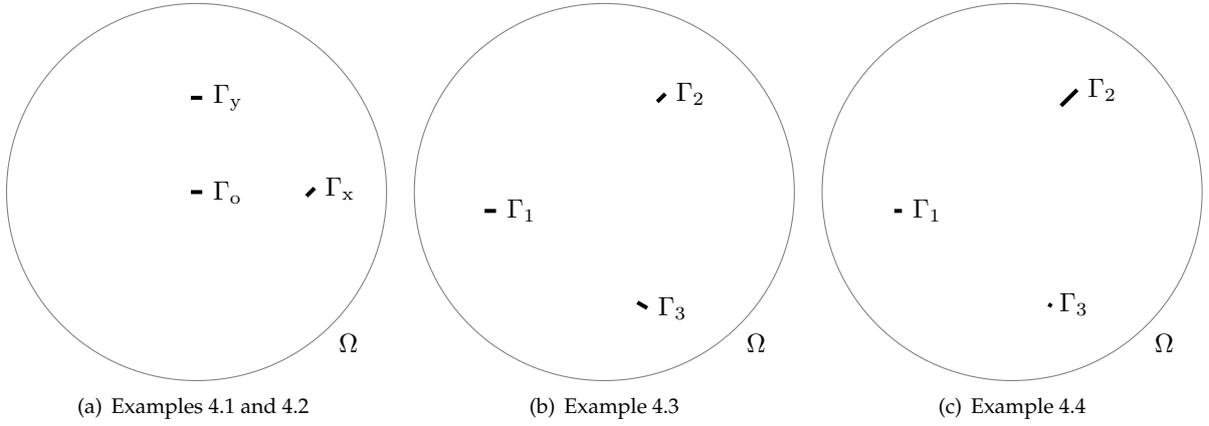

\begin{example}[Single crack centered at the origin]\label{Ex1}
We consider the identification of single crack located at $\mx_{\mathrm{o}}$:
\[\Gamma_{\mathrm{o}}=\set{(s,0):-0.03\leq s\leq0.03},\quad\mx_{\mathrm{o}}=(0,0).\]
Based on the imaging results in Figure \ref{Result1}, peak of magnitude $1$ was appeared at the center. Hence, by regarding the maximum value of $\mathfrak{F}(\mz;\ka)$, it is possible to recognize the exact location of $\Gamma_{\mathrm{o}}$ for any value of $\epsa$.
\end{example}

\begin{figure}[h]
\begin{center}
\subfigure[$\epsa=\epsb$]{\includegraphics[width=.33\textwidth]{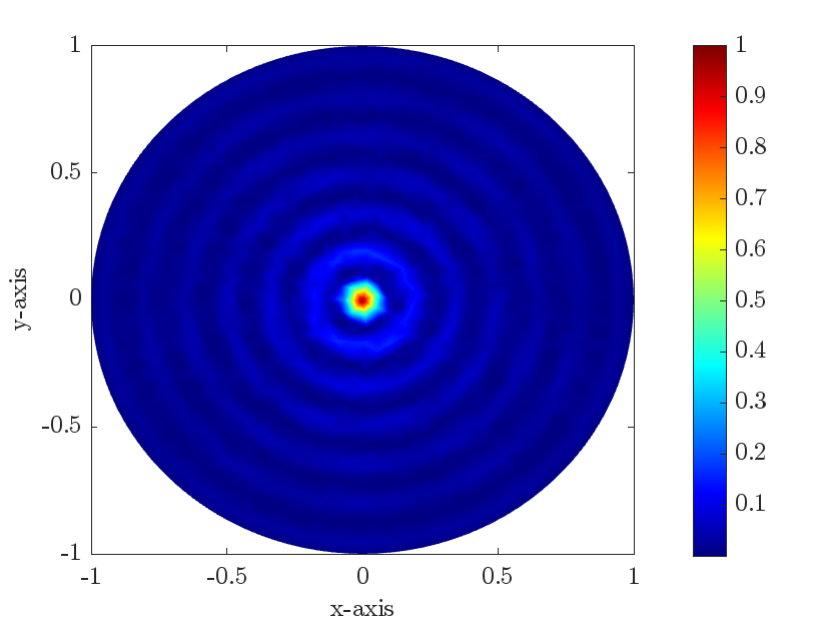}}\hfill
\subfigure[$\epsa=0.6\epsb$]{\includegraphics[width=.33\textwidth]{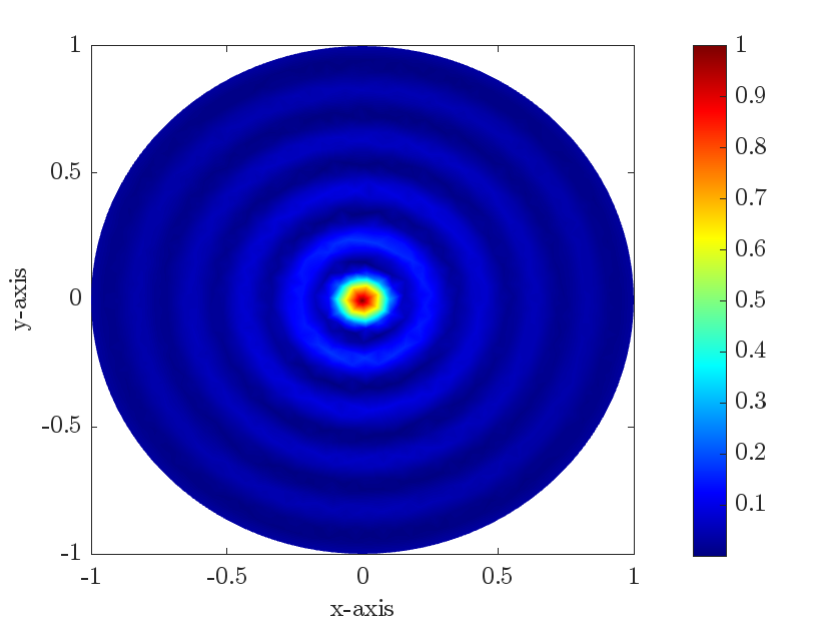}}\hfill
\subfigure[$\epsa=2.0\epsb$]{\includegraphics[width=.33\textwidth]{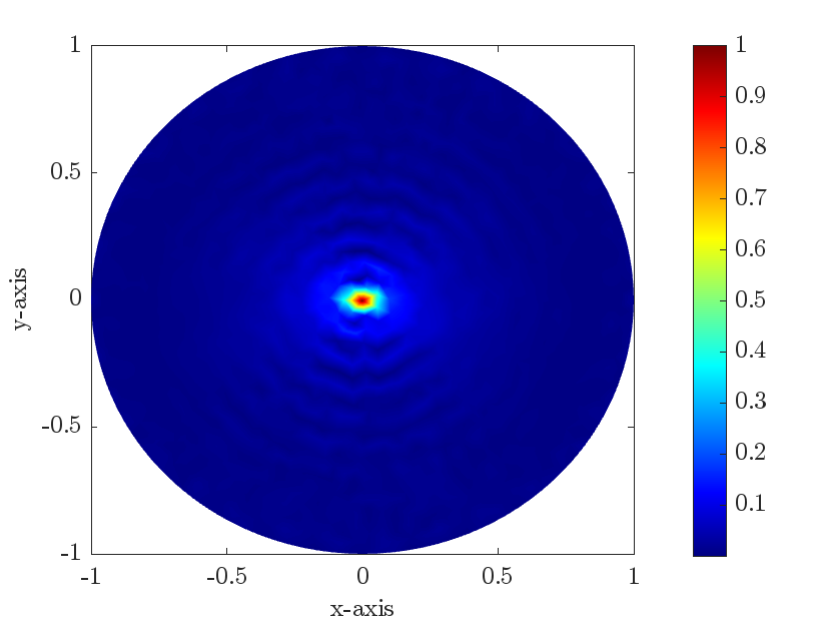}}
\caption{\label{Result1}(Example \ref{Ex1}) Maps of $\mathfrak{F}(\mz;\ka)$ with various $\epsa$ at $f=\SI{875}{\mega\hertz}$.}
\end{center}
\end{figure}

\begin{example}[Two cracks whose centers are located on the axis]\label{Ex2}
We consider the identification of two cracks with same length:
\begin{align*}
\Gamma_{\mathrm{x}}&=\set{\mathcal{R}_{\SI{45}{\degree}}(s+0.6,s):-0.03\leq s\leq0.03},&\mx_{\mathrm{x}}&=(0.6,0)\\
\Gamma_{\mathrm{y}}&=\set{(s,0.5):-0.03\leq s\leq0.03},&\mx_{\mathrm{y}}&=(0,0.5),
\end{align*}
where $\mathcal{R}_\theta$ denotes the rotation by $\theta$. Based on the imaging results in Figure \ref{Result2}, we can examine that $\Gamma_{\mathrm{x}}$ is located somewhere along the $x$-axis; that is, while the exact $y$-coordinate of the center can be determined, the $x$-coordinate cannot. Similarly, $\Gamma_{\mathrm{y}}$ is on located somewhere along the $y$-axis but its true location and shape cannot be identified. It is interesting to note that, if $\epsa=\epsb$, then the value of $\mathfrak{F}(\mz;\ka)$ is greater when $\mz$ lies on $\Gamma_{\mathrm{x}}$ than when it lies on $\Gamma_{\mathrm{y}}$ but in the case where $\epsa=1.2\epsb$, the situation is reversed. Finally, if $\epsa\geq1.5\epsb$, it is very difficult to distinguish between the $\Gamma_{\mathrm{y}}$ and several artifacts in the imaging results.
\end{example}

\begin{figure}[h]
\begin{center}
\subfigure[$\epsa=\epsb$]{\includegraphics[width=.33\textwidth]{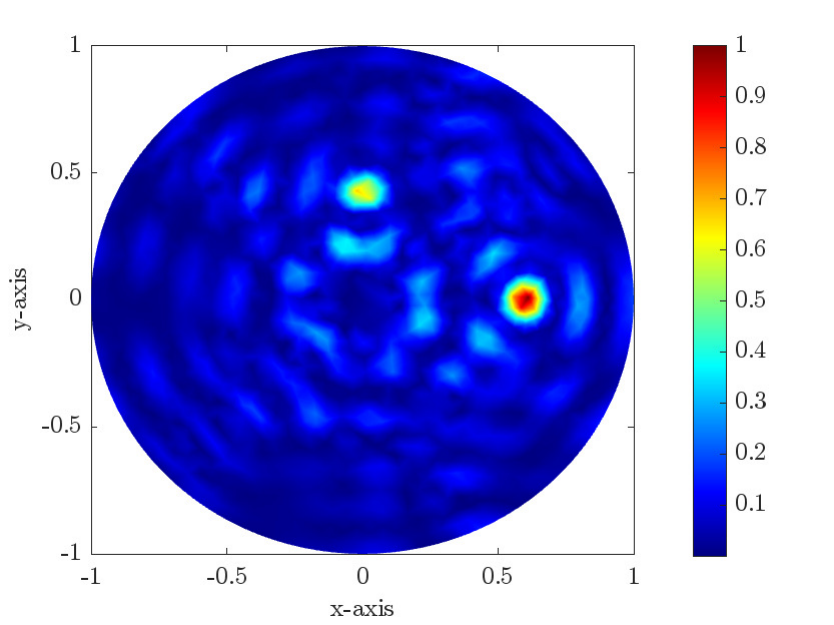}}\hfill
\subfigure[$\epsa=0.6\epsb$]{\includegraphics[width=.33\textwidth]{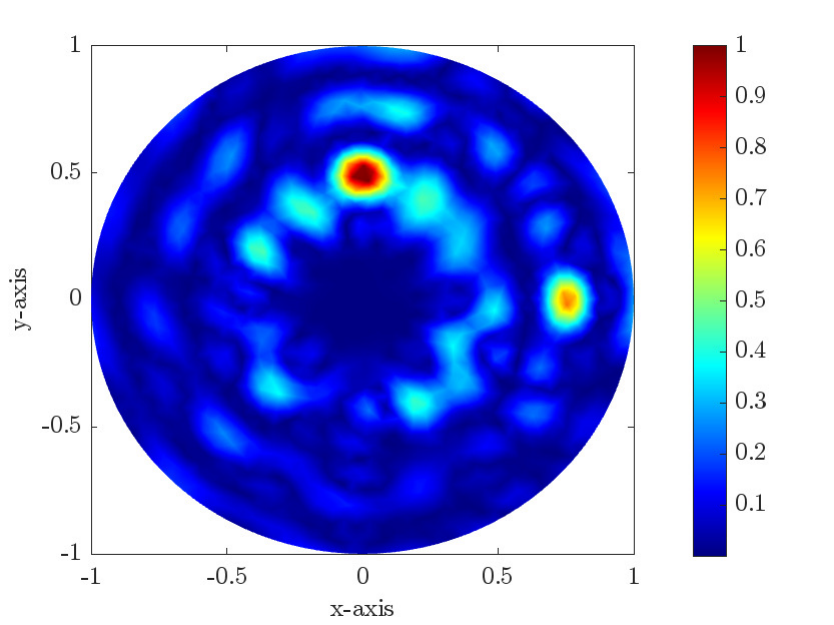}}\hfill
\subfigure[$\epsa=0.8\epsb$]{\includegraphics[width=.33\textwidth]{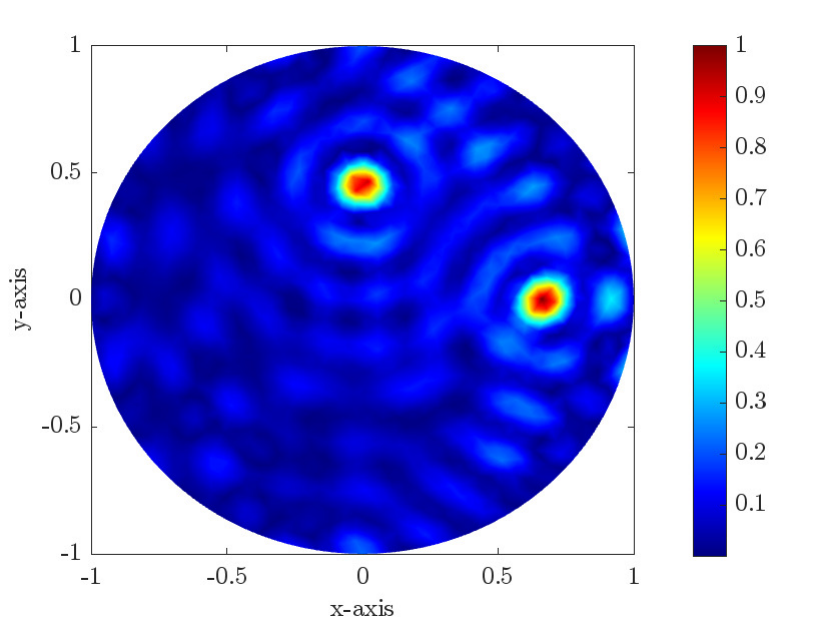}}
\\
\subfigure[$\epsa=1.2\epsb$]{\includegraphics[width=.33\textwidth]{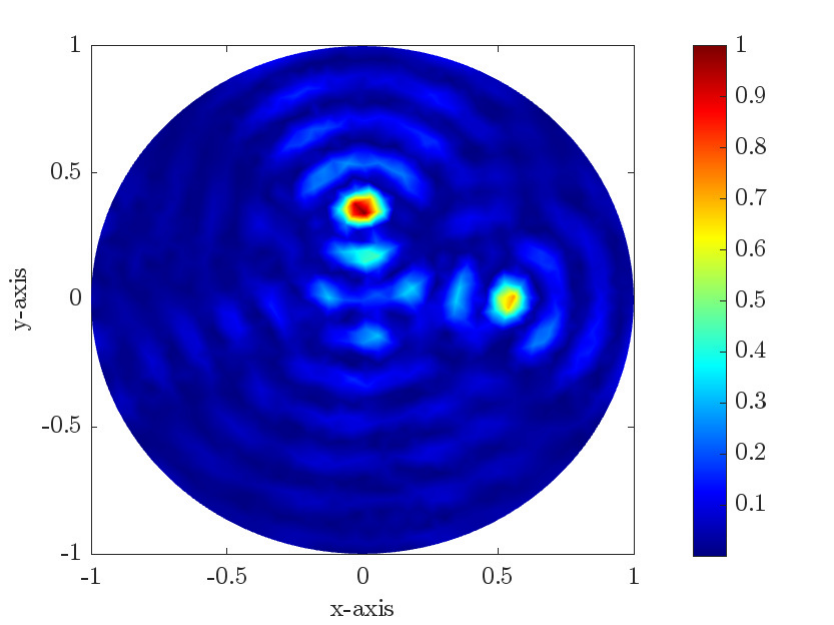}}\hfill
\subfigure[$\epsa=1.5\epsb$]{\includegraphics[width=.33\textwidth]{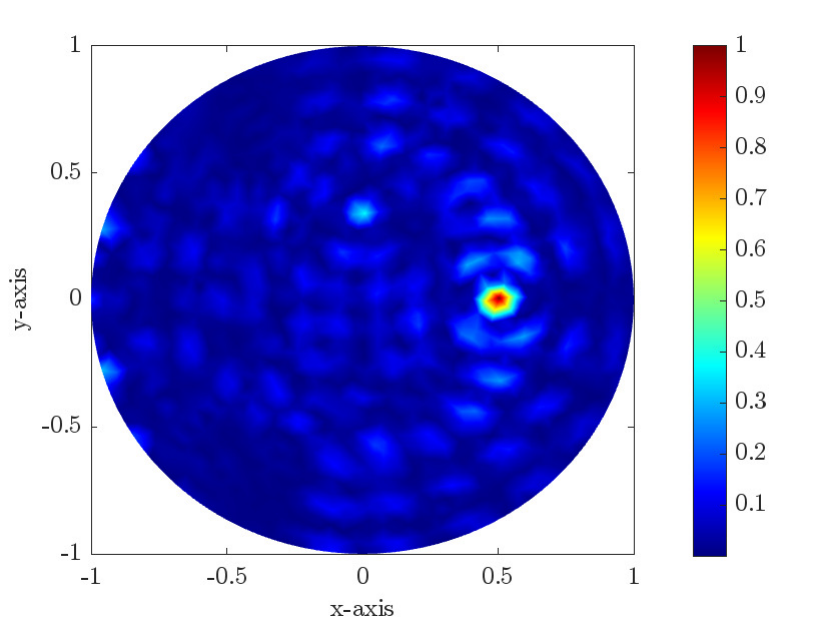}}\hfill
\subfigure[$\epsa=2.0\epsb$]{\includegraphics[width=.33\textwidth]{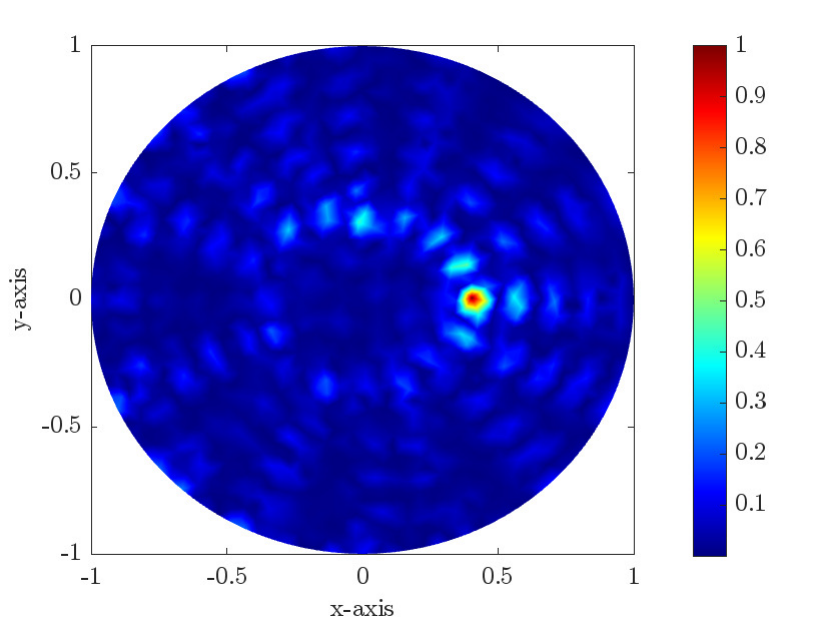}}
\caption{\label{Result2}(Example \ref{Ex2}) Maps of $\mathfrak{F}(\mz;\ka)$ with various $\epsa$ at $f=\SI{875}{\mega\hertz}$.}
\end{center}
\end{figure}

\begin{example}[Three cracks with same length]\label{Ex3}
We consider the identification of three cracks with same length:
\begin{align*}
\Gamma_1&=\set{(s-0.6,-0.1):-0.03\leq s\leq0.03},\quad&\mx_1&=(-0.6,-0.1)\\
\Gamma_2&=\set{\mathcal{R}_{\SI{45}{\degree}}(s+0.3,s+0.5):-0.03\leq s\leq0.03},\quad&\mx_2&=(0.3,0.5)\\
\Gamma_3&=\set{\mathcal{R}_{\SI{210}{\degree}}(s+0.2,s-0.6):-0.03\leq s\leq0.03},\quad&\mx_3&=(0.2,-0.6).
\end{align*}
Figure \ref{Result3} shows the maps of $\mathfrak{F}(\mz;\ka)$ with various $\epsa$. Based on the imaging results, if inaccurate value of $k$ was applied, it is impossible to obtain accurate information about the $x$ and $y$ components of $\mx_m$, $m=1,2,3$, i.e., exact location of $\Gamma_m$ cannot be retrieved. As we discussed in Property \ref{property2}, the identified location of each $\mx_m$ is close to the origin when $\epsa>\epsb$ and otherwise, the identified location of each $\mx_m$ is located far from the origin when $\epsa<\epsb$. Moreover, contrary to the theoretical result, although all $\Gamma_m$ have the same length, the value of $\mathfrak{F}(\mz;\ka)$ at $\mz\in\Gamma_3$ is smaller than at $\mz\in\Gamma_1$ or $\Gamma_2$ when $\epsa=0.6\epsb$ and $1.2\epsb$. Finally, similar to the results in Example \ref{Ex2}, it is very difficult to distinguish between the $\Gamma_3$ and several artifacts in the map of $\mathfrak{F}(\mz;2\epsb)$.
\end{example}

\begin{figure}[h]
\begin{center}
\subfigure[$\epsa=\epsb$]{\includegraphics[width=.33\textwidth]{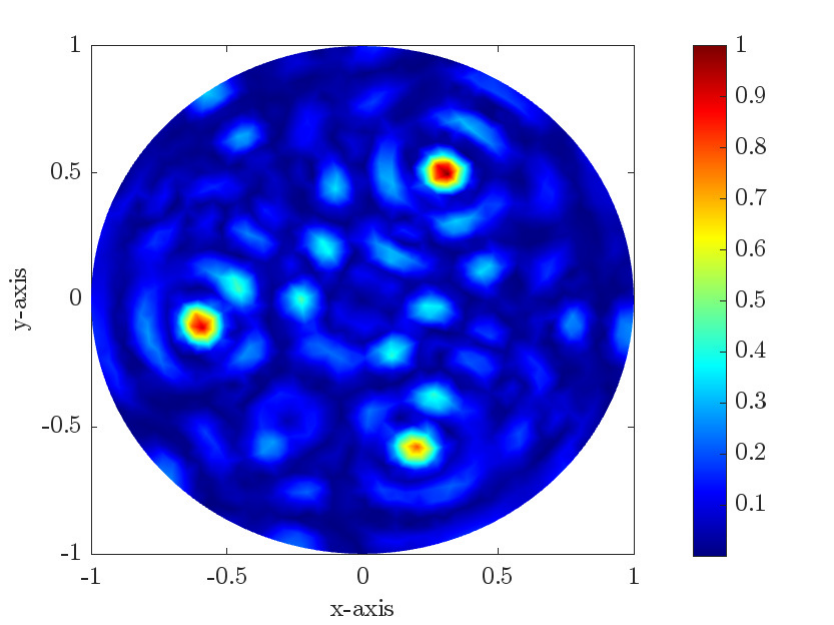}}\hfill
\subfigure[$\epsa=0.6\epsb$]{\includegraphics[width=.33\textwidth]{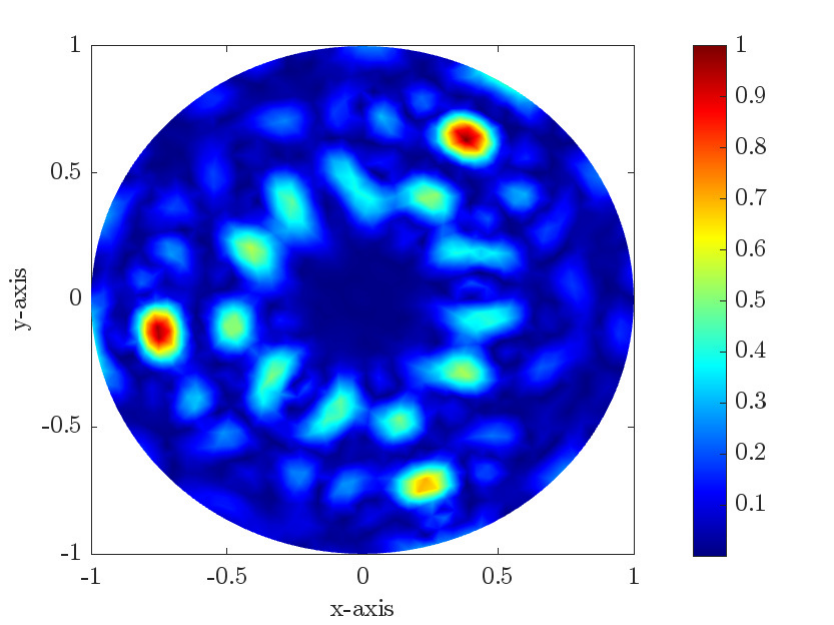}}\hfill
\subfigure[$\epsa=0.8\epsb$]{\includegraphics[width=.33\textwidth]{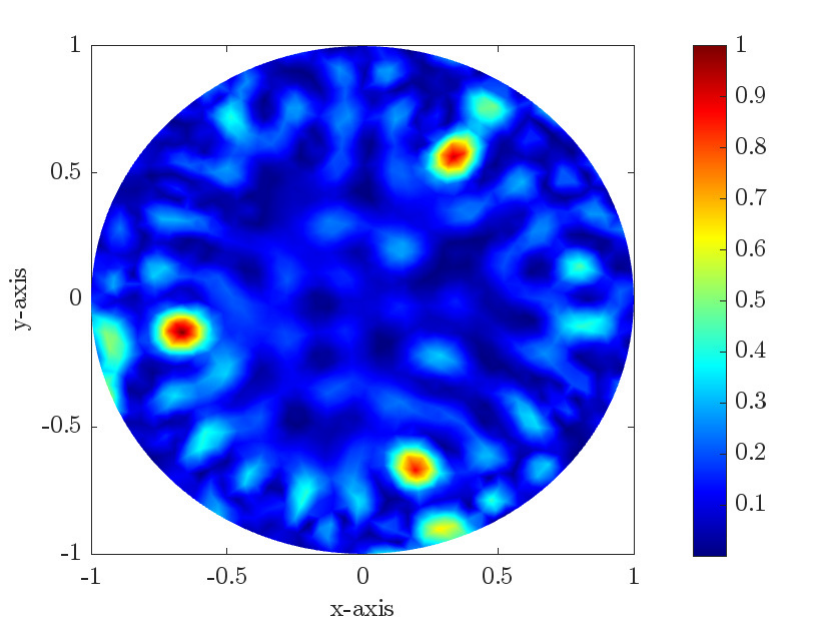}}\\
\subfigure[$\epsa=1.2\epsb$]{\includegraphics[width=.33\textwidth]{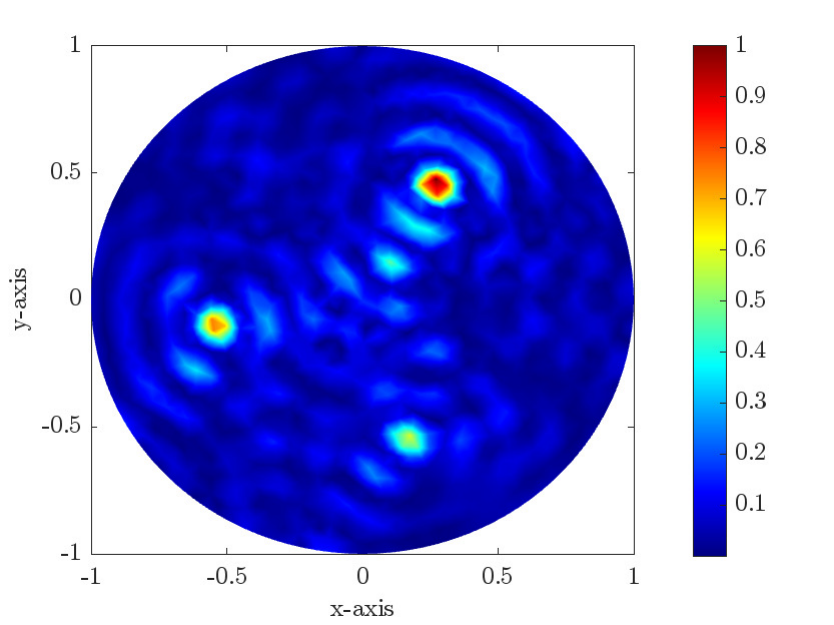}}\hfill
\subfigure[$\epsa=1.5\epsb$]{\includegraphics[width=.33\textwidth]{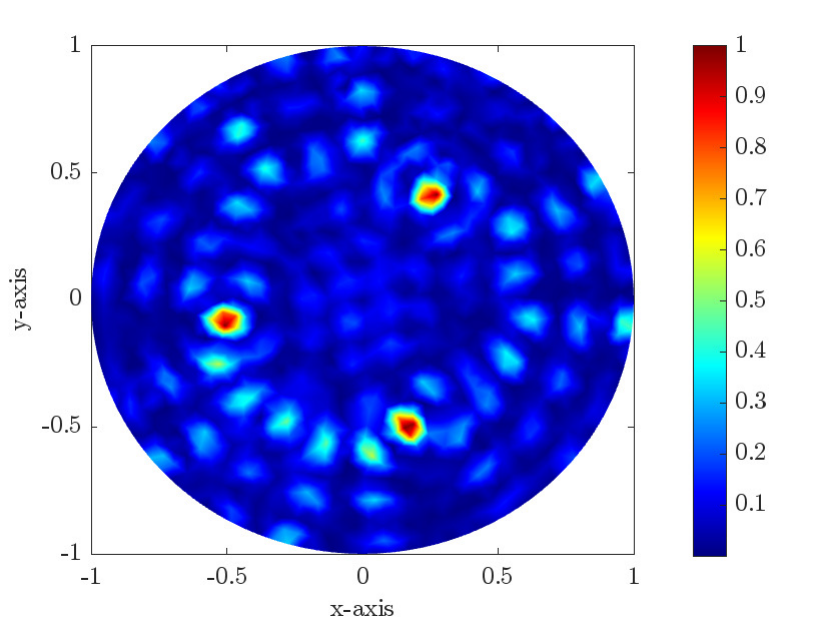}}\hfill
\subfigure[$\epsa=2.0\epsb$]{\includegraphics[width=.33\textwidth]{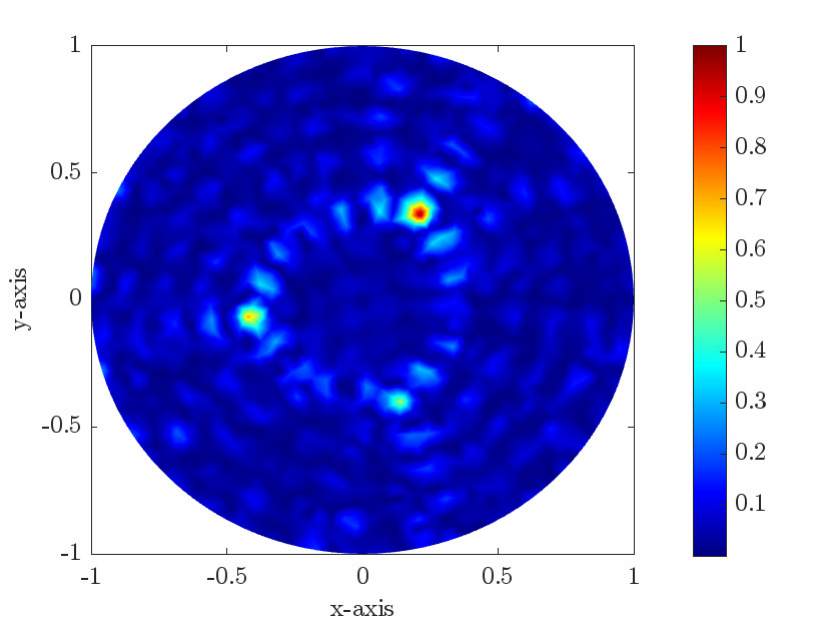}}
\caption{\label{Result3}(Example \ref{Ex3}) Maps of $\mathfrak{F}(\mz;\ka)$ with various $\epsa$ at $f=\SI{875}{\mega\hertz}$.}
\end{center}
\end{figure}

\begin{example}[Three cracks with difference lengths]\label{Ex4}
For the final example, let us consider the identification of three cracks with difference lengths:
\begin{align*}
\Gamma_1&=\set{(s-0.6,-0.1):-0.02\leq s\leq0.02},\quad&\mx_1&=(-0.6,-0.1)\\
\Gamma_2&=\set{\mathcal{R}_{\SI{45}{\degree}}(s+0.3,s+0.5):-0.06\leq s\leq0.06},\quad&\mx_2&=(0.3,0.5)\\
\Gamma_3&=\set{\mathcal{R}_{\SI{210}{\degree}}(s+0.2,s-0.6):-0.01\leq s\leq0.01}\quad&\mx_3&=(0.2,-0.6).
\end{align*}
Based on the imaging results in Figure \ref{Result4}, we can observe similar phenomena in Example \ref{Ex3}. Independently of this, it is observed that $\Gamma_3$ is shorter than $\Gamma_1$ and $\Gamma_2$, resulting in the value of $\mathfrak{F}(\mz;\ka)$ at $\mx_3$ being significantly smaller than at $\mx_1$ or $\mx_2$.
\end{example}

\begin{figure}[h]
\begin{center}
\subfigure[$\epsa=\epsb$]{\includegraphics[width=.33\textwidth]{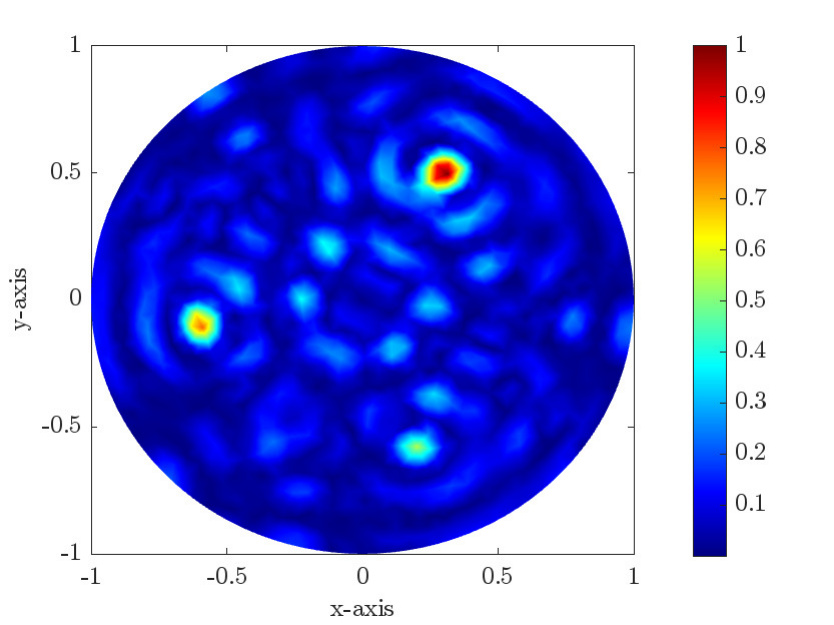}}\hfill
\subfigure[$\epsa=0.6\epsb$]{\includegraphics[width=.33\textwidth]{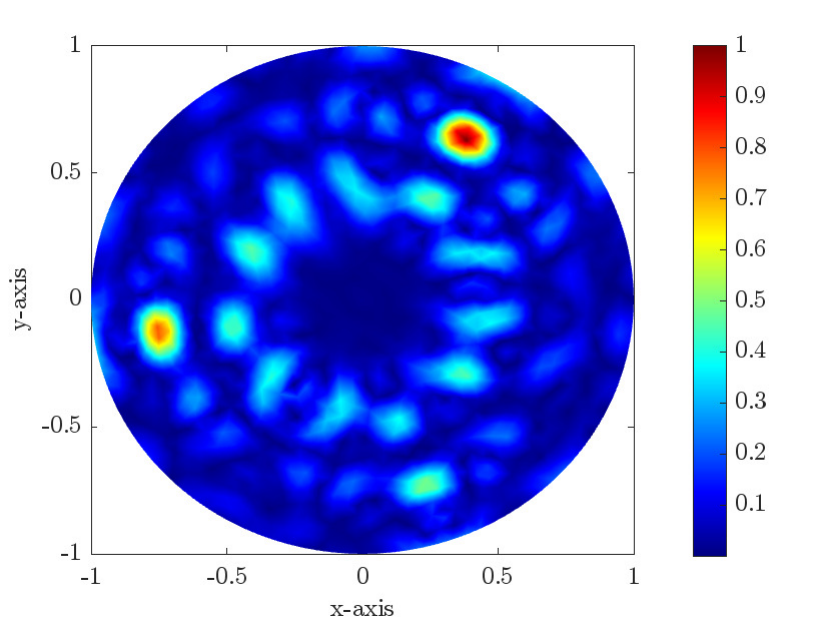}}\hfill
\subfigure[$\epsa=0.8\epsb$]{\includegraphics[width=.33\textwidth]{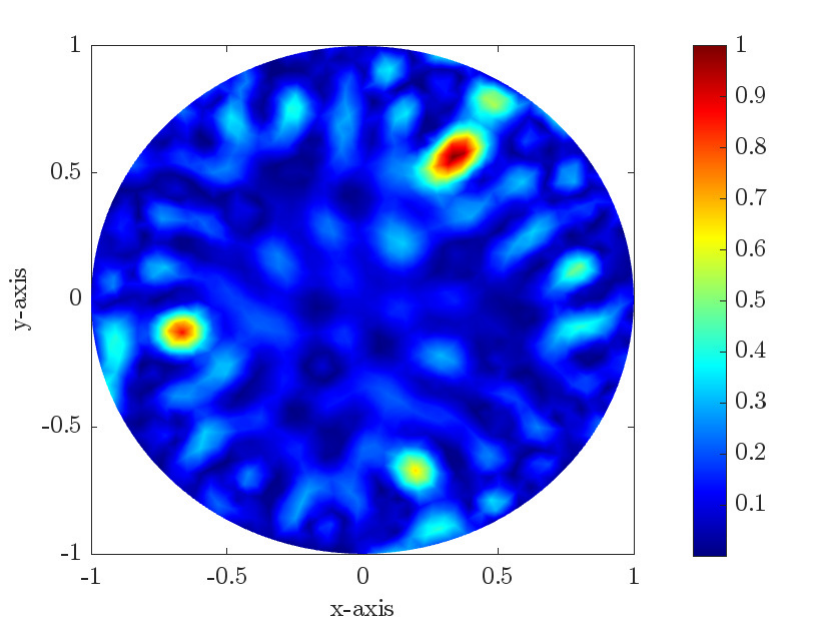}}\\
\subfigure[$\epsa=1.2\epsb$]{\includegraphics[width=.33\textwidth]{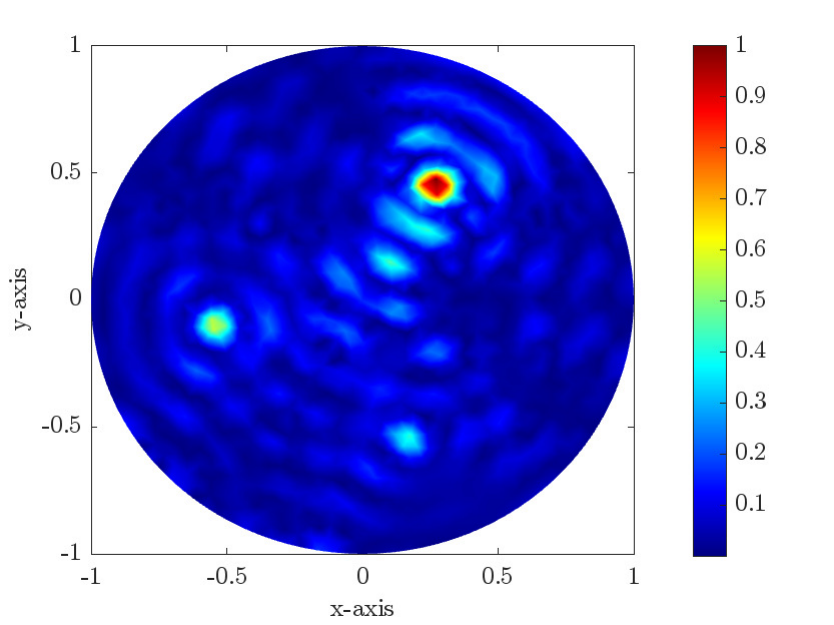}}\hfill
\subfigure[$\epsa=1.5\epsb$]{\includegraphics[width=.33\textwidth]{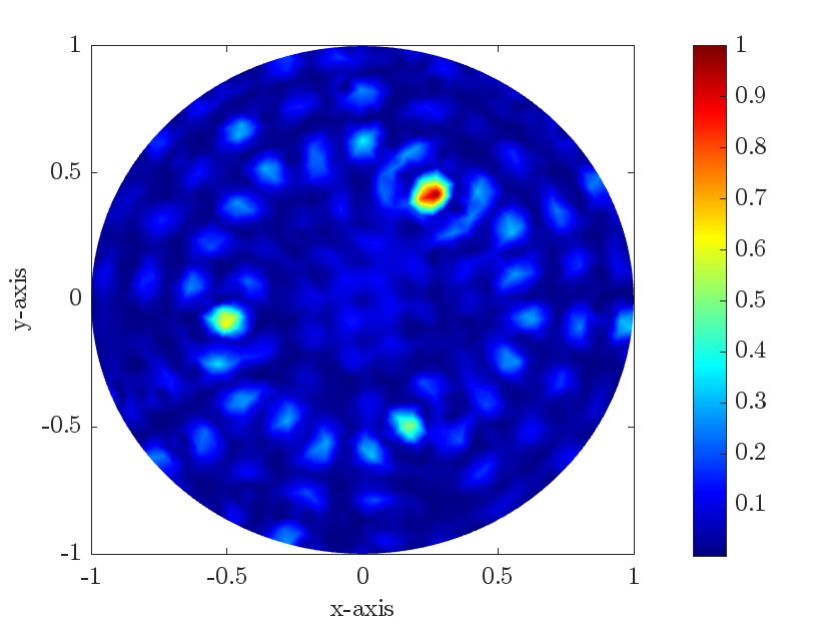}}\hfill
\subfigure[$\epsa=2.0\epsb$]{\includegraphics[width=.33\textwidth]{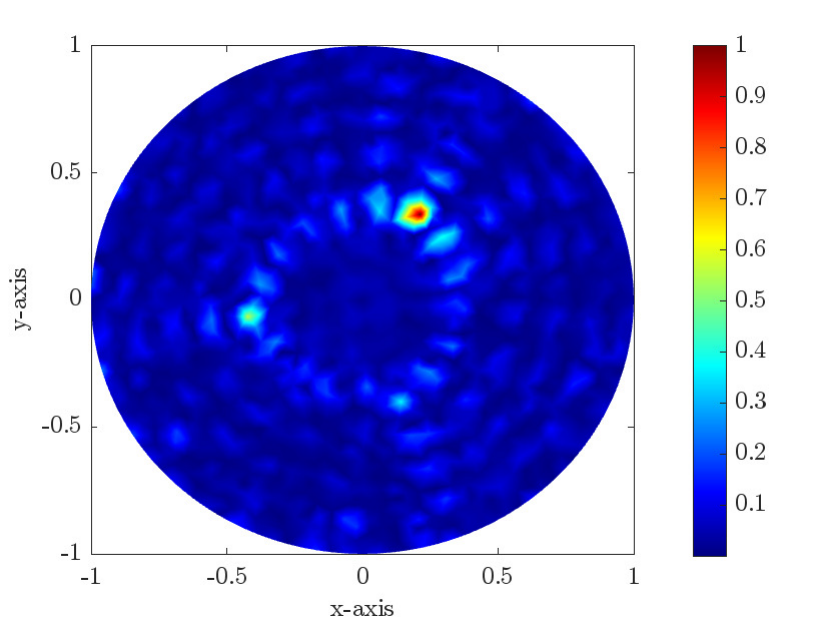}}
\caption{\label{Result4}(Example \ref{Ex4}) Maps of $\mathfrak{F}(\mz;\ka)$ with various $\epsa$ at $f=\SI{875}{\mega\hertz}$.}
\end{center}
\end{figure}

\section{Conclusion}\label{sec:5}
We investigated the applicability of topological derivative based imaging technique for a fast identification of a set of small, linear perfectly conducting cracks completely embedded in a homogeneous domain with smooth boundary when accurate values of background permittivity and/or permeability are unknown. Based on the asymptotic expansion formula for the far-field pattern in the presence of cracks, we rigorously showed that the imaging function of topological derivative can be expressed by the zero-order Bessel function and length of crack. Based on the discovered structure of the imaging function, we confirmed why inaccurate location of cracks was reconstructed when inaccurate value of background permittivity and/or permeability. To support the theoretical results, various numerical simulation results corrupted by random noise were exhibited.

In order to reconstruct cracks accurately, investigation of imaging algorithm for retrieving exact value of background permittivity and/or permeability would be a remarkable contribution to this study. In this paper, we considered an identification of two-dimensional cracks embedded in a homogeneous domain. Extension to the three-dimensional problem and real-world applications will be important research subjects.

\providecommand{\href}[2]{#2}
\providecommand{\arxiv}[1]{\href{http://arxiv.org/abs/#1}{arXiv:#1}}
\providecommand{\url}[1]{\texttt{#1}}
\providecommand{\urlprefix}{URL }

\end{document}